\documentclass[11pt]{article}
\usepackage[margin=3cm,a4paper]{geometry}
\usepackage{amsfonts,amsmath,amsthm,amscd,amssymb,latexsym,amsbsy,enumerate}

\newcommand{\R}{\mathbb R}
\newcommand{\abs}[1]{\left\vert#1\right\vert}
\newcommand{\norm}[1]{\left\lVert#1\right\rVert}
\newcommand{\norb}[1]{\big\lVert#1\big\rVert}
\newcommand{\set}[1]{\left\{#1\right\}}
\newcommand{\ex}[1]{\mathsf{E}\left[#1\right]}

\newcommand{\ind}[1]{\mathbf{1}_{#1}}
\newcommand{\wt}{\widetilde}

\begin{document}
\title{Integrability of solutions to mixed stochastic differential equations}
\author{Georgiy Shevchenko}

\theoremstyle{plain}
\newtheorem{theorem}{Theorem} 
\newtheorem{lemma}[theorem]{Lemma} 
\newtheorem{proposition}[section]{Proposition} 
\newtheorem{corollary}[section]{Corollary}
\newtheorem{definition}[section]{Definition}
\theoremstyle{definition}
\newtheorem{example}[section]{Example}
\newtheorem{remark}[section]{Remark}
\newcommand{\keywords}{\textbf{Keywors. }}
\newcommand{\subjclass}{\textbf{MSC 2010. }}
\renewcommand{\abstract}{\textbf{Abstract} }

\maketitle

\begin{abstract}
We prove that the standard conditions that provide unique solvability  of a mixed stochastic differential equations also guarantee that its solution possesses finite moments. We also present conditions supplying existence of exponential moments. For a special equation whose coefficients do not satisfy the linear growth condition, we find conditions for integrability of its solution.

\end{abstract}



\section*{Introduction}

The main object of this article is a stochastic differential equation of the form 
 \begin{equation}\label{mainsde}
 X_t =X_0 +\int_0^t a(s,X_s)ds +
\int_0^tb(s,X_s)dW_s+\int_0^t c(s,X_s) dZ_s,\quad t\in[0,T].
 \end{equation}
The randomness in this equation comes from two processes: a standard Wiener process  $W$ and a process $Z$ whose paths are H\"older continuous of order greater than $1/2$. In place of the process $Z$, usually a fractional Brownian motion $B^H$ with the Hurst parameter $H>1/2$ is taken. Due to such twofold nature of the randomness, equation \eqref{mainsde} is called a mixed stochastic differential equation. 

Existence and uniqueness of solution to a mixed stochastic diffential equation \eqref{mainsde} were proved under different conditions in \cite{guerra-nualart,kubilius,bookmyus,mbfbm-sde,mbfbm-limit}.
More generally, mixed equations with jumps were considered in \cite{mbfbm-jumps} and mixed delay equations, in \cite{mbfbm-delay}.

The principal aim of this article is to prove existence of moments  of a solution to \eqref{mainsde}. In  \cite{mbfbm-limit,mbfbm-jumps,mbfbm-delay}, the existence of moments was proved under an additional assumption of boundedness of the coefficient $b$. In \cite{mbfbm-malliavin}, the exponential integrability of solutions was established under the assumption that all coefficients of \eqref{mainsde} are bounded and certain other assumptions. In this paper we will generalize those results. Namely, we will show the existence of moments without any assumptions except those providing the unique solvability and cretain exponential integrability of the driver $Z$. Under additional assumption that the coefficients are bounded we show the exponential integrability of the solution to \eqref{mainsde}. We also consider an equation with coefficients not satisfying the linear growth condition and prove that all moments of its solution are finite. 

The paper is organized as follows.  In Section 1, we introduce the main object and provide necessary information on the pathwise (Young) integral. In Section 2 we show the usual and exponential integrability of the solution to \eqref{mainsde}. In Section 3 we prove existence of moment for a more general equation, whose coefficients do not satisfy the linear growth conditions.

\section{Preliminaries}\label{sec:prelim}

Let $\bigl(\Omega, \mathcal{F}, \mathbb{F} = \{\mathcal{F}_t, t\ge 0\},
\mathsf{P}\bigr)$ be a complete filtered probability space.  
We will use the following notation. The symbol $\abs{\cdot}$ will denote the absolute value of a real number, the Euclidean norm of a vector, and the (Euclidean) operator norm of a matrix. We will use the symbol $C$ to denote any constant whose value is not important and may vary from one line to another; should this constant depend on certain parameters, we will put them into subscripts. If the value of a constant is important, we will use the symbol $K$ for it. 

Now we proceed to a precise definition of the main object. It is the following stochastic differential equation in $\R^d$:
\begin{equation}\label{mainsde1}
X_t = X_0 + \int_{0}^{t} a(s,X_s)ds + \sum_{i=1}^{m} \int_0^t b_i(s,X_s)dW^i_s + \sum_{j=1}^{l}\int_0^t c_j(s,X_s)dZ^j_s,\quad t\in[0,T],
\end{equation}
where the coefficients $a\colon [0,T]\times \R^d\to\R^d$, $b_i\colon [0,T]\times \R^d\to\R^d$, $i=1,\dots,m$, $c_j\colon  [0,T]\times \R^d\to\R^d$, $j=1,\dots,l$, are jointly continuous; $W=\set{W_t=\left(W_t^1,\dots,W_t^m\right),t\in[0,T]}$ is a standard Wiener process in $\R^m$, $Z = \set{Z_t\left(Z_t^1,\dots,Z_t^l\right),t\in[0,T]}$ is an $\mathbb{F}$-adapted process in $\R^l$, whose paths are H\"older continuous of order $\mu>1/2$; the initial condition  $X_0$ is non-random. In what follows we will use the short form \eqref{mainsde} to write equation \eqref{mainsde1} and the integrals involved. 

In \eqref{mainsde}, the integral w.r.t.\ the Wiener process $W$ is understood as the It\^o integral, while that w.r.t.\ the process  $Z$, as the pathwise Young integral. We will give only basics on it; further information may be found e.g.\ in 
\cite{friz-victoir}. 

Let functions $g,h\colon [a,b]\to \R$ be $\alpha$- and $\beta$-H\"older continuous correspondingly, with $\alpha+\beta>1$. Then the integral $\int_a^b g(x) dh(x)$ is well defined as a limit of integral sums. Moreover, one has an estimate (the Young--Love inequality)
\begin{equation}
\left|\int_{a}^b g(s) dh(s)\right|\leq C_{\alpha,\beta} \norm{h}_{a,b,\beta}\big(\norm{g}_{a,b,\infty}+\norm{g}_{a,b,\alpha}(b-a)^{\alpha}\big)(b-a)^{\beta},
\label{integrestim}
\end{equation}
where $\norm{f}_{a,b,\infty}=\sup_{x\in [a,b]}\abs{f(x)}\ \text{ and }\ \norm{f}_{a,b,\gamma}=\sup_{a\leq s<t\leq b}\frac{|f(t)-f(s)|}{|t-s|^{\gamma}}
$
are the supremum norm and a $\gamma$-H\"older seminorm on $[a,b]$, respectively.

The following assumptions guarantee that equation \eqref{mainsde} has a unique solution, see \cite{mbfbm-delay}:
\begin{enumerate}[{A}1.]
\item
For all $t\in[0,T]$, $x\in \R^d$,
\begin{gather*}
\abs{a(t,x)} + \abs{b(t,x)}+ \abs{c(t,x)}\leq C(1+\abs{x}).
\end{gather*}
\item The function $c$ is differentiable in the second variable, moreover, the derivative is bounded: for all $t\in[0,T]$, $x\in \R^d$,
$$\abs{c'_x(t,x)}\le C.$$

\item For all $R>0$, $t\in[0,T]$ and $x_1,x_2\in\R^d$ such that $\abs{x_1}\le R$,  $\abs{x_2}\le R$, 
$$\abs{a(t,x_1)-a(t,x_2)}+\abs{b(t,x_1)-b(t,x_2)}+\abs{c'_x(t,x_1)-c'_x(t,x_2)}\leq C_R\abs{x_1-x_2}.$$
\item For some  $\beta\in(1-\mu,1/2)$ and any $s,t\in[0,T]$, $x\in\R^d$
$$\abs{c(t,x)-c(s,x)}\le C\abs{t-s}^\beta(1+\abs{x}), \quad\abs{c'_x(s,x)-c'_x(t,x)}\leq C|s-t|^\beta.$$
\end{enumerate}
Such formulation of the condition A4 is needed in order to be able to consider linear equations. 

\section{Integrability of solution}
In this section we prove integrability of the solution to equation \eqref{mainsde}. We use techniques similar to those used in \cite{hu-nualart,mbfbm-malliavin}. 
\begin{theorem}\label{thm-integr}
Assume that {\rm A1--A4} hold and
$$
\ex{\exp\set{c \norm{Z}^{1/\mu}_{0,T,\mu}}}<\infty.
$$
Then for any  $p>0$ the solution $X$ to equation \eqref{mainsde} satisfies
$$
\ex{\norm{X}_{0,T,\infty}^p}<\infty.
$$
\end{theorem}
\begin{proof}
For  $N \ge 1$, $R\ge 1$, denote $\tau_{N,R} = \min\set{t\ge 0: \norm{X}_{0,t,\infty}\ge R\text{ or } \norm{Z}_{0,t,\mu}\ge N}$, $X^{N,R}_t = X_{t\wedge \tau_{N,R}}$, $\ind{t} = \ind{\set{t\le \tau_{N,R}}}$. Put also
$I^a_t = \int_0^t a(s,X^{N,R}_s)\ind{s}ds$, $I^b_t = \int_0^t b(s,X^{N,R}_s)\ind{s}dW_s$, $I^c_t = \int_0^t c(s,X^{N,R}_s)\ind{s}dZ_s$. Fix arbitrary  $\theta\in(1-\mu, \beta]$.

Let $0\le s\le u\le v\le t\le T$. Write
\begin{align*}
\abs{X^{N,R}_v-X^{N,R}_u}&\le \abs{I^a_v-I^a_u} +\big|{I^b_v-I^b_u}\big|+ \abs{I^c_v-I^c_u}. \end{align*}
Estimate first
$$
\abs{I^a_v-I^a_u} \le \int_u^v \abs{a(z,X^{N,R}_z)} dz \le C\int_u^v \left(1+\abs{X^{N,R}_z}\right) dz \le C\left(1+\norm{X^{N,R}}_{s,t,\infty}\right)(v-u).
$$
Further, using \eqref{integrestim}, we have
\begin{align*} \abs{I^c_v-I^c_u} \le C N  \left(
\norm{c(\cdot,X^{N,R}_\cdot)}_{u,v,\infty} + 
\norm{c(\cdot,X^{N,R}_\cdot)}_{u,v,\theta}(v-u)^\theta\right)(v-u)^\mu
\end{align*}
It follows from assumption A1 that  $$\norm{c(\cdot,X^{N,R}_\cdot)}_{u,v,\infty}\le C \left(1+\norm{X^{N,R}}_{u,v,\infty}\right)\le C \left(1+\norm{X^{N,R}}_{s,t,\infty}\right).$$ 
Since
\begin{align*}
\abs{c(x,X^{N,R}_x)-c(y,X^{N,R}_y)}&\le \abs{c(x,X^{N,R}_x)-c(y,X^{N,R}_x)} + \abs{c(y,X^{N,R}_x)-c(y,X^{N,R}_y)} \\ &\le C\left(\abs{x-y}^\beta(1+\abs{X^{N,R}_x}) + \abs{X^{N,R}_x - X^{N,R}_y} \right),
\end{align*}
then
$$
\norm{c(\cdot,X^{N,R}_\cdot)}_{u,v,\theta}\le C\left((v-u)^{\beta-\theta}\big(1+\norm{X^{N,R}}_{u,v,\infty}\big) + \norm{X^{N,R}}_{u,v,\theta}\right).
$$
Therefore, 
\begin{align*}
\abs{I^c_v-I^c_u}
\le CN \left(1+\norm{X^{N,R}}_{s,t,\infty} + \norm{X^{N,R}}_{s,t,\theta} (v-u)^\theta \right)(v-u)^\mu.
\end{align*}
Consequently, 
\begin{align*}
\norm{X^{N,R}}_{s,t,\theta}&\le C\left(1+\norm{X^{N,R}}_{s,t,\infty}\right)(t-s)^{1-\theta}
+ \norb{I^b}_{s,t,\theta}\\
&\qquad  + CN  \left(1+\norm{X^{N,R}}_{s,t,\infty}(t-s)^{\mu-\theta} + \norm{X^{N,R}}_{s,t,\theta} (t-s)^\mu \right)\\
&\le \norb{I^b}_{s,t,\theta} + K N\left(1+\norm{X^{N,R}}_{s,t,\infty}(t-s)^{\mu-\theta} + \norm{X^{N,R}}_{s,t,\theta} (t-s)^\mu \right)
\end{align*}
with certain non-random constant $K $. 

Suppose that $t-s\le\Delta$ with $\Delta \le (2K  N)^{-1/\mu}$. Then
\begin{equation}\label{xnr}
\norm{X^{N,R}}_{s,t,\theta}\le 2\norb{I^b}_{s,t,\theta} + 2K N\left(1+\norm{X^{N,R}}_{s,t,\infty}(t-s)^{\mu-\theta} \right).
\end{equation}
Further, prom the obvious inequality
$$\norm{X^{N,R}}_{s,t,\infty}\le \abs{X_s} + \norm{X^{N,R}}_{s,t,\theta} (t-s)^\theta,
$$
using \eqref{xnr}, we obtain
\begin{align*}
\norm{X^{N,R}}_{s,t,\infty}&\le \norm{X^{N,R}}_{0,s,\infty} +  2\left(\norb{I^b}_{s,t,\theta} + K N\right)(t-s)^\theta+2KN\norm{X^{N,R}}_{s,t,\infty} (t-s)^\mu\\
&\le \norm{X^{N,R}}_{0,s,\infty} +  2\left(\norb{I^b}_{s,t,\theta} + K N\right)\Delta^\theta+2KN\norm{X^{N,R}}_{s,t,\infty} \Delta^\mu,
\end{align*}
whenever $t-s\le \Delta$. Assuming further that $\Delta\le (4K  N)^{-1/\mu}$, we get
$$
\norm{X^{N,R}}_{s,t,\infty}\le 2\norm{X^{N,R}}_{0,s,\infty} +  4\left(\norb{I^b}_{s,t,\theta}+KN\right)\Delta^\theta.
$$
Hence we derive for any $p>(1/2-\theta)^{-1}$ that
\begin{equation}\label{exxnrp}
\ex{\norm{X^{N,R}}^p_{0,t,\infty}}\le C_p \left(\ex{\norm{X^{N,R}}_{0,s,\infty}^p} + \ex{\norb{I_b}_{s,t,\theta}^p}\Delta^{p\theta}+N^p\Delta^{p\theta}\right).
\end{equation}
Using the Garsia--Rodemich--Rumsey inequality, we have
\begin{align*}
\ex{\norb{I_b}_{s,t,\theta}^p}&\le C_{p} \int_s^t\int_s^t \frac{\ex{\abs{I^b(x)-I^b(y)}^p}}{\abs{x-y}^{p\theta +2}}dx\,dy\\
& \le C_{p} \int_s^t\int_s^t \ex{\abs{\int_x^y \abs{b(z,X^{N,R}_z)}^2\ind{z} dz}^{p/2}}{\abs{x-y}^{-p\theta -2}}dx\,dy\\
& \le C_{p} \int_s^t\int_s^t \int_x^y \left(1+\ex{\abs{X^{N,R}_z}^{p}}\right) dz \abs{x-y}^{p/2-p\theta -3}dx\,dy\\
&\le C_p\left(1+ \ex{\norm{X^{N,R}}_{s,t,\infty}^p}\right)\int_s^t \int_s^t \abs{x-y}^{p/2-p\theta} dx\,dy\\
&\le  C_{p} \left(1+ \ex{\norm{X^{N,R}}_{s,t,\infty}^p}\right)\Delta^{p/2-p\theta}.
\end{align*}
Plugging this estimate into \eqref{exxnrp}, we arrive at the inequality
\begin{align*}
\ex{\norm{X^{N,R}}_{0,t,\infty}^p}&\le K_{p} \left(\ex{\norm{X^{N,R}}_{0,s,\infty}^p} + \ex{\norm{X^{N,R}}_{s,t,\infty}^p}\Delta^{p/2}+N^p\Delta^{p\theta}\right)\\
&\le K_p\left(\ex{\norm{X^{N,R}}_{0,s,\infty}^p} + \ex{\norm{X^{N,R}}_{0,t,\infty}^p}\Delta^{p/2}+N^p\right)
\end{align*}
with certain constant $K_{p}$. Assuming that $\Delta\le (2K_{p})^{-2/p}$, we get
\begin{equation}\label{exnrp}
\ex{\norm{X^{N,R}}_{0,t,\infty}^p} \le 2K_{p} \left(\ex{\norm{X^{N,R}}_{0,s,\infty}^p} +N^p\right).
\end{equation}
Finally, put $\Delta = \min\set{(4KN)^{-1/\mu},(2K_{p})^{-2/p}}$. Splitting the segment $[0,T]$ into $[T/\Delta]+1$ parts of length at most $\Delta$, we obtain from the estimate \eqref{exnrp} that
\begin{equation*}
\ex{\norm{X^{N,R}}_{0,T,\infty}^p} \le (2K_{p}+1)^{T/\Delta + 1} \left(\abs{X_0}^p +N^p\right)\le 
C_p\exp\set{C_p N^{1/\mu}}.
\end{equation*}
Letting  $R\to\infty$ and using the Fatou lemma, we get
$$
\ex{\norm{X}_{0,T,\infty}^p\ind{\norm{Z}_{0,T,\mu}\le N}} \le 
K'_p\exp\set{K'_p N^{1/\mu}}
$$
with some constant $K'_p$.
Denote $\xi = \norm{X}_{0,T,\infty}^p$, $\eta = \norm{Z}_{0,T,\mu}$ and write
\begin{align*}
\left(\ex{\xi^p}\right)^2 &\le \ex{\exp\set{2K'_{2p}\eta^{1/\mu}}}\ex{\xi^{2p}\exp\set{-2K'_{2p} \eta^{1/\mu}}}\\&\le C_p \sum_{n=1}^\infty \ex{\xi^{2p}\exp\set{-2K'_{2p} \eta^{1/\mu}}\ind{\eta\in[n-1,n)}}\\
&\le C_p\sum_{n=1}^{\infty}\exp\set{-2K_{2p}'(n-1)^{1/\mu}} \ex{\xi^{2p}\ind{\eta\in[n-1,n)}}\\
&\le C_p\sum_{n=1}^{\infty}\exp\set{-2K_{2p}'(n-1)^{1/\mu}}\exp\set{K_{2p}'n^{1/\mu}}<\infty,
\end{align*}
as required.
\end{proof}
An improtant particular example of equation \eqref{mainsde} is an equation involving fractional Brownian motion. Recall that an $l$-dimensional fractional Brownian motion $H\in(0,1)$ is a centered Gaussian process $B^H = \set{B_t^H = (B_t^{H,1},\dots,B_t^{H,l}),t\in[0,T]}$ with the covariance function
$$
\ex{B^{H,i}_tB^{H,j}_s} = \frac{\delta_{ij}}{2}\left(t^{2H}+s^{2H}-\abs{t-s}^{2H}\right).
$$
It is well known that a fractional Brownian motion has a version which satisfies the H\"older condition with any exponent  $\mu<H$. We assume henceforth that this version is taken.

\begin{corollary}
Let in equation \eqref{mainsde} the process $Z=B^H$ be a fractional  Brownian motion with a parameter $H\in(1/2,1)$. 
Then for any $p>0$
$$
\ex{\norm{X}_{0,T,\infty}^p}<\infty.$$
\end{corollary}
\begin{proof}
Let $\mu\in (1/2,H)$. 
Then the H\"older seminorm
$$
\norm{B^H}_{0,T,\mu} = \sup_{0\le s<t\le T}\frac{\abs{B^H_t-B^H_s}}{(t-s)^{\mu}}
$$
is an almost surely finite supremum of norms of a centered Gaussian family. Therefore, by Fernique's theorem,
$
\ex{\exp\set{a\norm{B^H}_{0,T,\mu}^2}}<\infty
$
for some $a>0$. Since $\mu>1/2$, then
$
\ex{\exp\set{c\norm{B^H}_{0,T,\mu}^{1/\mu}}}<\infty
$
for all $c>0$. Thus, the required statement follows from Theorem~\ref{thm-integr}.
\end{proof}

The exponential integrability will be proved under a different set of assumptions. Some assumptions are carried forward unchanged, nevertheless we repeat them for convenience. 
\begin{enumerate}[{B}1.]
\item
For all $t\in[0,T]$, $x\in \R^d$,
\begin{gather*}
\abs{a(t,x)} + \abs{b(t,x)}+ \abs{c(t,x)}\leq C.
\end{gather*}
\item For all $t\in[0,T]$, $x\in \R^d$,
$$\abs{c'_x(t,x)}\le C.$$
\item For all $R>0$, $t\in[0,T]$ and $x_1,x_2\in\R^d$ such that $\abs{x_1}\le R$,  $\abs{x_2}\le R$,
$$\abs{a(t,x_1)-a(t,x_2)}+\big|{b(t,x_1)-b(t,x_2)}\big|+\abs{c'_x(t,x_1)-c'_x(t,x_2)}\leq C_R\abs{x_1-x_2}.$$
\item For all  $s,t\in[0,T]$, $x\in\R^d$,
$$\abs{c(t,x)-c(s,x)}+ \abs{c'_x(t,x)-c'_x(s,x)}\leq C|s-t|^\beta.$$
\end{enumerate}

Under these assumption the exponential integrability of the solution to \eqref{mainsde} is proved the same way as it is made in \cite{mbfbm-malliavin} for coefficients independent of $t$. Nevertheless, for completeness we will give principal ideas, omitting unimportant details.
\begin{theorem}\label{thm-expintegr}
Assume that the assumptions {\rm B1--B4} are satisfied and for any
$c>0$, $\alpha\in(0,2)$
$$
\ex{\exp\set{c \norm{Z}^{\alpha}_{0,T,\mu}}}<\infty.
$$
Then for any $c>0$, $\gamma\in(0,4\mu/(2\mu+1))$ the solution $X$ to equation \eqref{mainsde} satisfies
$$
\ex{\exp\set{c\norm{X}_{0,T,\infty}^\gamma}}<\infty.
$$
\end{theorem}
\begin{proof}
The proof partially repeats that of Theorem~\ref{thm-integr}, so some details will be left out.

Denote $I^a_t = \int_0^t a(s,X_s)ds$, $I^b_t = \int_0^t b(s,X_s)dW_s$, $I^c_t = \int_0^t c(s,X_s)dZ_s$ and fix arbitrary $\theta\in(1-\mu,\beta]$, $\kappa\in(\theta,1/2)$.

Let $0\le s\le u\le v\le t\le T$. Write
\begin{align*}
\abs{X_v-X_u}&\le \abs{I^a_v-I^a_u} +\abs{I^b_v-I^b_u}+ \abs{I^c_v-I^c_u}.
\end{align*}
Estimate
$$
\abs{I^a_v-I^a_u} \le \int_u^v \abs{a(z,X_z)} dz \le C(v-u).
$$
Further, using \eqref{integrestim}, we have
\begin{align*}
\abs{I^c_v-I^c_u}
\le C \norm{Z}_{0,t,\mu}\left(1+  \norm{X}_{s,t,\theta} (v-u)^\theta \right)(v-u)^\mu.
\end{align*}
Evidently, $\abs{I^b_v-I^b_u}\le \norb{I^b}_{s,t,\kappa}(v-u)^{\kappa}$. The estimates above yield
\begin{align*}
\norm{X}_{s,t,\theta}&\le C(t-s)^{1-\theta}
+ \norb{I^b}_{s,t,\kappa}(t-s)^{\kappa-\theta}  + C\norm{Z}_{0,t,\mu}\left((t-s)^{\mu-\theta}  + \norm{X}_{s,t,\theta} (t-s)^\theta \right) \\&\le \norb{I^b}_{s,t,\kappa}(t-s)^{\kappa-\theta} + K (1+\norm{Z}_{0,t,\mu})\left((t-s)^{\mu-\theta} + \norm{X}_{s,t,\theta} (t-s)^\mu \right)
\end{align*}
with a positive constant $K$. Assuming that $t-s\le\Delta  := (2K  (1+\norm{Z}_{0,T,\mu}))^{-1/\mu}$, we have
\begin{equation}\label{xsttheta}
\norm{X}_{s,t,\theta}\le 2\norb{I^b}_{s,t,\kappa}(t-s)^{\kappa-\theta} + 2K (1+\norm{Z}_{0,t,\mu})(t-s)^{\mu-\theta}.
\end{equation}
As in the proof of Theorem~\ref{thm-integr}, the last estimate implies
\begin{align*}
\norm{X}_{0,t,\infty}&\le \norm{X}_{0,s,\infty} +  2 \norb{I^b}_{s,t,\kappa}(t-s)^\kappa + 2K
(1+\norm{Z}_{0,t,\mu})(t-s)^{\mu}\\
& \le \norm{X}_{0,s,\infty} +  2K\big(\norb{I^b}_{s,t,\kappa}\Delta^\kappa +
(1+\norm{Z}_{0,t,\mu})\Delta^{\mu}\big).
\end{align*}
Splitting the segment  $[0,T]$ into  $[T/\Delta]+1$ parts of length at most $\Delta$, we get
\begin{align*}
\norm{X}_{0,T,\infty} &\le \abs{X_0} + 2K(T+1)\left(\norb{I^b}_{0,T,\kappa}\Delta^{\kappa-1} + 2K
(1+\norm{Z}_{0,T,\mu})\Delta^{\mu-1}\right)\\
& \le  C\left(1 + \norb{I^b}_{0,T,\kappa}(1+\norm{Z}_{0,T,\mu})^{(1-\kappa)/\mu} + (1+\norm{Z}_{0,T,\mu})^{1/\mu}\right)
\end{align*}
Now take arbitrary $\gamma\in(0,4\mu/(2\mu+1))$.  Since
$(2-\gamma)\mu > \gamma/2$, then it is possible to choose  $\kappa$ so that $1-\kappa\in(1/2,(2-\gamma)\mu/\gamma)$, equivalently,
$$
\frac{2(1-\kappa)}{(2-\gamma)\mu}<\frac2\gamma.
$$
Now take any $\lambda>2/(2-\gamma)$ so that $\nu:=\lambda(1-\kappa)/\mu<2/\gamma$, and denote $\lambda'=\lambda/(1-\lambda)$ the adjoint exponent for $\lambda$; from $\lambda>2/(2-\gamma)$ it follows that $\lambda'<2/\gamma$. From the Young inequality
\begin{align*}
\norb{I^b}_{0,T,\kappa}(1+\norm{Z}_{0,T,\mu})^{(1-\kappa)/\mu} \le
\frac{1}{\lambda'}\norb{I^b}_{0,T,\kappa}^{\lambda'} + \frac{1}{\lambda}(1+\norm{Z}_{0,T,\mu})^{\nu}.
\end{align*}
Therefore,
\begin{align*}
\norm{X}_{0,T,\infty}^\gamma &\le  C\left(1 + \norb{I^b}_{0,T,\kappa}^{\lambda'} + (1+\norm{Z}_{0,t,\mu})^{\nu} + (1+\norm{Z}_{0,T,\mu})^{1/\mu}\right)^\gamma\\
&\le C\left(1 + \norb{I^b}_{0,T,\kappa}^{\lambda'\gamma} + (1+\norm{Z}_{0,T,\mu})^{\nu\gamma} + (1+\norm{Z}_{0,T,\mu})^{\gamma/\mu}\right).
\end{align*}
Hence the statement of the theorem follows, because the exponents are less than $2$ and for $\alpha\in(0,2)$ $\ex{\exp\set{c\norm{Z}_{0,T,\mu}^\alpha}}<\infty$ by the assumption, $\ex{\exp\set{c\norm{I^b}_{0,T,\kappa}^\alpha}}<\infty$ by \cite[Lemma 1]{mbfbm-malliavin}.
\end{proof}
\begin{corollary}
Let in \eqref{mainsde} the process $Z$ is a fractional Brownian motion $B^H$ with the Hurst parameter $H\in(1/2,1)$, and the coefficients satisfy the assumptions {\rm B1--B4}. 
Then for all $c>0$, $\gamma\in(0,4H/(2H+1))$ the solution $X$ to equation \eqref{mainsde} satisfies
$$
\ex{\exp\set{c\norm{X}_{0,T,\infty}^\gamma}}<\infty.
$$
\end{corollary}

\section{Integrability of solution for equations without linear growth condition}
Consider now equation of a form
$$
Y_t = Y_0 + \int_0^t \tilde{a}(s,X_s,Y_s)ds + \sum_{i=1}^r\int_0^t  \tilde{b}_i(s,X_s,Y_s)d\wt W^{i}_t + \sum_{j=1}^q  \int_0^t \tilde{c}_j(s,X_s,Y_s)d\wt Z^{j}_s,\quad t\in[0,T],
$$
where $X$ solves \eqref{mainsde}; the coefficients $\tilde{a}\colon [0,T]\times \R^d\times \R^k\to\R^k$, $\tilde b_i\colon [0,T]\times \R^d\times \R^k\to\R^k$, $i=1,\dots,r$, $\tilde c_j\colon  [0,T]\times \R^d\times \R^k\to\R^k$, $j=1,\dots,q$, are jointly continuous; $\wt W=\set{\wt W_t=\big(\wt W_t^{1},\dots,\wt W_t^{r}\big),t\in[0,T]}$ is a standard Wiener process in $\R^r$, $\wt Z = \set{\wt Z_t=\big(\wt Z_t^{1},\dots,\wt Z_t^{q}\big),t\in[0,T]}$ is an $\mathbb{F}$-adapted process in $\R^q$ having $\mu$-H\"older continuous paths; the initial condition $Y_0$ is non-random. Abbreviate this equation as
\begin{equation}\label{mainsde2}
Y_t = Y_0 + \int_0^t \tilde{a}(s,X_s,Y_s)ds + \int_0^t  \tilde{b}(s,X_s,Y_s)d\wt W_t +  \int_0^t \tilde{c}(s,X_s,Y_s)d\wt Z_s.
\end{equation}
Such equations arise in modeling quite often. For instance, in financial mathematics, a price process in a stochastic volatility model can be driven by an equation
\begin{equation}\label{stochvol}
S_t = S_0 + \int_0^t \mu_s S_u du + \int_0^t \sigma^W_u S_u dW_u + \int_0^t \sigma^B_u S_u dB^H_u,
\end{equation}
where the stochastic volatility processes  $\sigma^W$ and $\sigma^B$ are also solutions to some stochastic differential equations. Another example is the equation satisfied by the Malliavin derivative of the solution to \eqref{mainsde}:
\begin{equation}\label{stochder}
d DX_t = a'(X_t)DX_t dt + b'(X_t)DX_t dW_t + c'(X_t)DX_t dZ_t.
\end{equation}
If we combine equation \eqref{stochvol}  with volatility equations or equation \eqref{stochder} with $\eqref{mainsde}$, then the coefficients of resulting multi-dimensional equation, generally speaking, will not satisfy the linear growth condition. So we need some other techniques to study the integrability.

In out case the role of  `stochastic volatility' is played by the solution $X$ to \eqref{mainsde}. We will assume that the coefficients to \eqref{mainsde} satisfy the assumptions B1--B4.
We formulate the assumptions on the coefficients \eqref{mainsde2}, using \eqref{stochvol} and \eqref{stochder} as model equations. Specifically, we will assume that for some $\rho\in[0,2/3)$
\begin{enumerate}[C1.]
\item For all $t\in[0,T]$, $x\in\R^d$, $y\in\R^k$,
$$
\abs{\tilde{a}(t,x,y)} + \abs{\tilde{c}(t,x,y)} \le C(1+\abs{x}^\rho)(1+\abs{y}).
$$
\item For all $t\in[0,T]$, $x\in\R^d$, $y\in\R^k$ 
$$
 \big|{\tilde{b}(t,x,y)}\big| \le C(1+\abs{y}).
$$
\item
For all $t\in[0,T]$, $x\in\R^d$, $y\in\R^k$ 
$$
\abs{\tilde{c}'_y(t,x,y)}\le C(1+\abs{x}^\rho).
$$
\item For all $R>1$, $t\in[0,T]$ and $x\in\R^d$, $y_1,y_2\in\R^k$ such that $\abs{x}\le R$, $\abs{y_1}\le R$, $\abs{y_2}\le R$,
\begin{align*}
&\abs{\tilde{a}(t,x,y_1)-\tilde{a}(t,x,y_2)} + \big|{\tilde{b}(t,x,y_1)-\tilde{b}(t,x,y_2)}\big|  \\&\qquad\qquad + \abs{\tilde{c}'_y(t,x,y_1)- \tilde{c}'_y(t,x,y_2)}\le C\abs{y_1-y_2}.
\end{align*}
\item For all  $t\in[0,T]$, $x_1,x_2\in\R^d$, $y\in\R^k$,
\begin{align*}
&\abs{\tilde{c}(t,x_1,y)-\tilde{c}(t,x_2,y)} \le C\abs{x_1-x_2}(1+\abs{y}).
\end{align*}
\item For all $s,t\in[0,T]$, $x\in\R^d$, $y\in\R^k$,
\begin{align*}
&\abs{\tilde{c}(s,x,y)-\tilde{c}(t,x,y)}\le 
C\abs{s-t}^\beta(1+\abs{x}^\rho)(1+\abs{y}),\\
&\abs{\tilde{c}'_y(s,x,y)-\tilde{c}'_y(t,x,y)}\le C\abs{s-t}^\beta(1+\abs{y}).
\end{align*}
\end{enumerate}
Unfortunately, we were not able to prove existence of \textit{all} moments under the assumption $|\tilde b(t,x,y)|\le C(1+\abs{x}^\rho)(1+\abs{y})$, so we impose C2. 

The proof of the unique solvability for the equation \eqref{mainsde2} under assumptions C1--C6 is similar to that for equation \eqref{mainsde} (see \cite{mbfbm-delay}), so we omit it.

\begin{theorem}\label{thm-last}
Assume that the coefficients of equation \eqref{mainsde} satisfy  {\rm B1--B4}, and the coefficients of \eqref{mainsde2} satisfy {\rm C1--C6} with  $\rho\in(0, 2\mu(2\mu-1)/(2\mu+1))$. Let also for any $c>0$, $\alpha\in(0,2)$
$$
\ex{\exp\set{c \norm{Z}^{\alpha}_{0,T,\mu}}}+ \ex{\exp\set{c \norb{\wt Z}^{\alpha}_{0,T,\mu}}}<\infty.
$$
Then for any $p>0$ the solution $Y$ to \eqref{mainsde2} satisfies
$$
\ex{\norm{Y}_{0,T,\infty}^p}<\infty.
$$
\end{theorem}
\begin{remark}
The restriction $\rho< 2\mu(2\mu-1)/(2\mu+1)$ explains why  $\rho<2/3$ in C1--C6: the right-hand side of the former inequality increases in $\mu$ and is equal to $2/3$ for $\mu=1$.
\end{remark}
\begin{proof}
The proof will follow the same scheme as the proofs of Theorem~\ref{thm-integr} and \ref{thm-expintegr}.

Put $J^b_t = \int_0^t b(s,X_s)dW_s$. Fix arbitrary $N\ge 1$, $M\ge 1$, $R\ge 1$,   $\theta\in(1-\mu, \beta]$, $\kappa\in(\theta,1/2)$.  Denote $$\tau = \min\set{t\ge 0: \norm{Y}_{0,t,\infty}\ge R\text{ or } \norm{Z}_{0,t,\mu}+\norb{J^b}_{0,t,\kappa}\ge N\text{ or }\norm{X}^\rho_{0,t,\infty}\ge M},$$
$X^\tau_t = X_{t\wedge \tau}$, $Y^{\tau}_t = Y_{t\wedge \tau}$, $\ind{t} = \ind{\set{t\le \tau}}$. Put also
$I^a_t = \int_0^t \tilde{a}(s,X^{\tau}_s,Y^{\tau}_s)\ind{s}ds$, $I^b_t = \int_0^t \tilde{b}(s,X^{\tau}_s,Y^{\tau}_s)\ind{s}d\wt W_s$, $I^c_t = \int_0^t \tilde{c}(s,X^{\tau}_s,Y^{\tau}_s)\ind{s}d\wt Z_s$.

Let $0\le s\le u\le v\le t\le T$. Write
\begin{align*}
\abs{Y^{\tau}_v-Y^{\tau}_u}&\le \abs{I^a_v-I^a_u} +\big|{I^b_v-I^b_u}\big|+ \abs{I^c_v-I^c_u}. \end{align*}
From the condition C1
\begin{align*}
\abs{I^a_v-I^a_u} &\le \int_u^v \abs{\tilde{a}(z,X^{\tau}_z,Y^{\tau}_z)}\ind{z} dz \le CM\left(1+\norm{Y^{\tau}}_{s,t,\infty}\right)(v-u).
\end{align*}
Denoting $\xi_t = \tilde c(t,X^{\tau}_t,Y^{\tau}_t)$, we have from \eqref{integrestim} that
\begin{align*} \abs{I^c_v-I^c_u} \le C N  \left(
\norm{\xi}_{u,v,\infty} + 
\norm{\xi}_{u,v,\theta}(v-u)^\theta\right)(v-u)^\mu.
\end{align*}
The assumption C1 allows to estimate
 $$\norm{\xi}_{u,v,\infty}\le C  M\left(1+\norm{Y^{\tau}}_{s,t,\infty}\right).$$ 
The inequalities
\begin{align*}
&\abs{\tilde{c}(x,X^{\tau}_x,Y^{\tau}_x)-\tilde{c}(y,X^{\tau}_x,Y^{\tau}_x)}\le \abs{\tilde{c}(x,X^{\tau}_x,Y^{\tau}_x)-\tilde{c}(y,X^{\tau}_x,Y^{\tau}_x)} \\
& \qquad\qquad +
\abs{\tilde{c}(y,X^{\tau}_x,Y^{\tau}_x)-\tilde{c}(y,X^{\tau}_y,Y^{\tau}_x)} + \abs{\tilde{c}(y,X^{\tau}_y,Y^{\tau}_x)-\tilde{c}(y,X^{\tau}_y,Y^{\tau}_y)} \\
&\le C\left(\abs{x-y}^\beta\left(1+\abs{X^{\tau}_x}^\rho\right)\left(1+\abs{Y^{\tau}_x}\right) + \abs{X^{\tau}_x - X^{\tau}_y}\left(1+\abs{Y^{\tau}_x}\right)+\abs{Y^{\tau}_x - Y^{\tau}_y}\left(1+\abs{X^{\tau}_y}^\rho\right) \right)
\end{align*}
imply that
\begin{align*}
&\norm{\xi}_{u,v,\theta}\le CM\big((v-u)^{\beta-\theta}\big(1+\norm{Y^{\tau}}_{u,v,\infty}\big)
 + \norm{Y^{\tau}}_{u,v,\theta}\big)+  \norm{X^\tau}_{u,v,\theta}\big(1+\norm{Y^{\tau}}_{u,v,\infty}\big).
\end{align*}
From these inequalities, we obtain
\begin{align*}
&\norm{Y^{\tau}}_{s,t,\theta}\le CM\big(1+\norm{Y^{\tau}}_{s,t,\infty}\big)(t-s)^{1-\theta}
+ \norb{I^b}_{s,t,\theta}  + CMN \big(1 + \norm{Y^{\tau}}_{s,t,\infty}\big)  (t-s)^{\mu-\theta} \\&\qquad
+ CN\Big(M\big((t-s)^{\beta-\theta}\big(1+ \norm{Y^{\tau}}_{s,t,\infty}\big) + \norm{Y^{\tau}}_{s,t,\theta}\big)+  \norm{X^\tau}_{s,t,\theta}\big(1+\norm{Y^{\tau}}_{s,t,\infty}\big)\Big)(t-s)^{\mu}\\
& \le \norb{I^b}_{s,t,\theta} + CMN \left(\big(1+\norm{Y^{\tau}}_{s,t,\infty}\big)(t-s)^{\mu-\theta} + \norm{Y^{\tau}}_{s,t,\theta}(t-s)^{\mu}\right) \\&\qquad
+ CN \norm{X^\tau}_{s,t,\theta}\big(1+\norm{Y^{\tau}}_{s,t,\infty}\big)(t-s)^{\mu}.
\end{align*}
Hence we get, similarly to \eqref{xsttheta}, 
$$
\norm{X^\tau}_{s,t,\theta}\le CN(t-s)^{\kappa-\theta} + C N(t-s)^{\mu-\theta}\le CN(t-s)^{\kappa-\theta}.
$$
Consequently, 
\begin{align*}\norm{Y^{\tau}}_{s,t,\theta} &\le \norb{I^b}_{s,t,\theta} + KMN \left(\big(1+\norm{Y^{\tau}}_{s,t,\infty}\big)(t-s)^{\mu-\theta} + \norm{Y^{\tau}}_{s,t,\theta}(t-s)^{\mu}\right) \\&\qquad\qquad
+ KN^2 \big(1+\norm{Y^{\tau}}_{s,t,\infty}\big)(t-s)^{\mu+\kappa-\theta}
\end{align*}
with some constant $K$.

Assume that $t-s\le\Delta$ with $\Delta \le (2K M N)^{-1/\mu}$. Then
\begin{align*}
\norm{Y^{\tau}}_{s,t,\theta}&\le 2\norb{I^b}_{s,t,\theta} + 2KMN \big(1+\norm{Y^{\tau}}_{s,t,\infty}\big)(t-s)^{\mu-\theta} \\
&\qquad\qquad + 2KN^2 \big(1+\norm{Y^{\tau}}_{s,t,\infty}\big)(t-s)^{\mu+\kappa-\theta}.
\end{align*}
Hence, as in Theorems~\ref{thm-integr} and \ref{thm-expintegr}, we have
\begin{align*}
\norm{Y^{\tau}}_{s,t,\infty}&\le 2\norb{I^b}_{s,t,\theta}\Delta^\theta + \norm{Y^{\tau}}_{0,s,\infty} +  2KNM\big(1+\norm{Y^{\tau}}_{s,t,\infty}\big)\Delta^{\mu}\\
&\qquad\qquad + 2KN^2 \big(1+\norm{Y^{\tau}}_{s,t,\infty}\big)\Delta^{\mu+\kappa}.
\end{align*}
for $t-s\le \Delta$. Assume further that $\Delta\le \min\set{(8K  N M)^{-1/\mu}, (8KN^2)^{-1/(\mu+\kappa)}}$; then
$$
\norm{Y^{\tau}}_{s,t,\infty}\le 2\norm{Y^{\tau}}_{0,s,\infty}  + 4\norb{I^b}_{s,t,\theta}.
$$
Therefore, for arbitrary $p>(1/2-\theta)^{-1}$
\begin{equation*}
\ex{\norm{Y^{\tau}}_{0,t,\infty}}\le C_p \left(\ex{\norm{X^{\tau}}_{0,s,\infty}^p} + \ex{\norb{I_b}_{s,t,\theta}^p}\Delta^{p\theta}+N^p\Delta^{p\theta}\right).
\end{equation*}
Using the same reasoning as in the proof of Theorem~\ref{thm-integr}, we arrive at the estimate
\begin{align*}
\ex{\norm{Y^{\tau}}_{0,t,\infty}^p}\le  K_p\left(\ex{\norm{Y^{\tau}}_{0,s,\infty}^p} + \ex{\norm{Y^{\tau}}_{0,t,\infty}^p}\Delta^{p/2}+N^p\right),
\end{align*}
with some positive constant $K_{p}$. Putting $$\Delta = \min\set{(8K  N M)^{-1/\mu}, (8KN^2)^{-1/(\mu+\kappa)}, (2K_{p})^{-2/p}},$$ we get
\begin{equation*}
\ex{\norm{Y^{\tau}}_{0,t,\infty}^p} \le 2K_{p} \left(\ex{\norm{Y^{\tau}}_{0,s,\infty}^p} +N^p\right),
\end{equation*}
therefore,
\begin{align*}
\ex{\norm{Y^{\tau}}_{0,T,\infty}^p} \le (2K_{p}+1)^{T/\Delta + 1} \left(\abs{Y_0}^p +N^p\right)\le 
C_p\exp\set{C_p \big(N^{1/\mu}M^{1/\mu}+N^{2/(\mu+\kappa)}\big)}.
\end{align*}
As in the proof of Theorem~\ref{thm-integr}, we derive hence that
$$
\left(\ex{\norm{Y}_{0,T,\infty}^p}
\right)^2 
\le C_p \exp\set{C_p \left(\xi^{1/\mu}\eta^{1/\mu} + \xi^{2/(\mu+\kappa)}, \right)}
$$
where $\xi = \norm{Z}_{0,t,\mu}+\norm{J_b}_{0,t,\kappa}$, $\eta = \norm{X}_{0,T,\infty}^\rho$. Since $\mu+\kappa>1$, then for any $c>0$ \ $\ex{\exp\set{c\xi^{2/(\mu+\kappa)}}}<\infty$. From the restriction on $\rho$ it follows that $2\mu(2\mu-1)^{-1}< 4\mu^2\rho^{-1}(2\mu+1)^{-1}$. Choose arbitrary $\lambda\in (2\mu(2\mu-1)^{-1}, 4\mu^2\rho^{-1}(2\mu+1)^{-1})$, denote $\lambda' = \lambda/(\lambda-1)$ the exponent adjoint to $\lambda$ and write by the Young inequality
$$
\xi^{1/\mu}\eta^{1/\mu}\le C(\xi^{\lambda'/\mu} + \eta^{\lambda/\mu}).
$$
Theorem~\ref{thm-expintegr} implies that $\ex{\exp\set{c\eta^{\lambda/\mu}}}<\infty$ for any $c>0$. It is easy to see that  $\lambda'<2\mu$, so $\ex{\exp\set{c\xi^{\lambda'/\mu}}}<\infty$ for any $c>0$. Thus, the theorem is proved. \end{proof}
\begin{corollary}
Let in equations \eqref{mainsde} and \eqref{mainsde2} the processes  $Z$ and $\wt Z$ be fractional Brownian motionswith thte Hurst parameter $H\in(1/2,1)$, and let the coefficients of the equations satisfy assumptions {\rm B1--B4} and {\rm C1--C6} with $\rho\in(0,2H(2H-1)/(2H+1))$, correspondingly.  Then for any $p>0$ the solution $Y$ to equation \eqref{mainsde2} satisfies
$$
\ex{\norm{Y}_{0,T,\infty}^p}<\infty.
$$
\end{corollary}
\begin{remark}
The last corollary allows to deduce that the solution to \eqref{mainsde} with $Z=B^H$ has an integrable Malliavin derivative provided that the coefficients are differentiable, the derivative of $b$ is bounded, and the derivatives of $a$ and $c$ grow slower than a  power function with an exponent less than $(0,2H(2H-1)/(2H+1))$.
\end{remark}


\begin{thebibliography}{10}

\bibitem{friz-victoir}
P.~K. Friz, N.~B. Victoir.
\newblock {\em Multidimensional stochastic processes as rough paths}, volume
  120 of {\em Cambridge Studies in Advanced Mathematics}. Theory and applications.
\newblock Cambridge University Press, Cambridge, 2010.
\newblock xiv+656 p.


\bibitem{guerra-nualart}
J.~Guerra, D.~Nualart.
\newblock Stochastic differential equations driven by fractional {B}rownian
  motion and standard {B}rownian motion.
\newblock {\em Stoch. Anal. Appl.}, 26(5), 2008, p.~289--315.

\bibitem{hu-nualart}
Y.~Hu, D.~Nualart.
\newblock Differential equations driven by {H}\"older continuous functions of
  order greater than 1/2.
\newblock In {\em Stochastic analysis and applications}, volume~2 of {\em Abel
  Symp.},  Springer, Berlin, 2007, p.~399--413.

\bibitem{kubilius}
K.~Kubilius.
\newblock The existence and uniqueness of the solution of an integral equation
  driven by a {$p$}-semimartingale of special type.
\newblock {\em Stochastic Process. Appl.}, 98(2), 2002, p.~289--315.

\bibitem{bookmyus}
Y.~S. Mishura.
\newblock {\em Stochastic calculus for fractional {B}rownian motion and related
  processes}, volume 1929 of {\em Lecture Notes in Mathematics}.
\newblock Springer-Verlag, Berlin, 2008. 
\newblock xviii+393 p.

\bibitem{mbfbm-sde}
Y.~S. Mishura and G.~M. Shevchenko.
\newblock Stochastic differential equation involving {W}iener process and
  fractional {B}rownian motion with {H}urst index {$H> 1/2$}.
\newblock {\em Comm. Statist. Theory Methods}, 40(19--20), 2011, p.~3492--3508.

\bibitem{mbfbm-limit}
Y.~S. Mishura and G.~S. Shevchenko.
\newblock Mixed stochastic differential equations with long-range dependence:
  {E}xistence, uniqueness and convergence of solutions.
\newblock {\em Comput. Math. Appl.}, 64(10), 2012, p.~3217--3227.

\bibitem{mbfbm-jumps}
G.~M. Shevchenko.
\newblock Mixed fractional stochastic differential equations with jumps.
\newblock {\em Stochastics}, 2013.
\newblock {A}rticle in press.

\bibitem{mbfbm-delay}
G.~M. Shevchenko.
\newblock Mixed stochastic delay differential equations.
\newblock To appear in {\em Theory Probab. Math.
Stat.}, 2013.
\newblock arXiv:math.PR/1306.0590.

\bibitem{mbfbm-malliavin}
G.~M. Shevchenko and T.~O. Shalaiko.
\newblock Malliavin regularity of solutions to mixed stochastic differential
  equations.
\newblock {\em Stat. Probab. Letters}, 83(12), 2013, p.~2638-2646.
\end{thebibliography}
\end{document}